\documentclass[12pt,a4paper,reqno]{amsart}

\usepackage[lmargin=1.4 in, rmargin= 1.4in, tmargin=1.7in]{geometry}

\title{Deformations of quasi-categories in modules}

\author{Violeta Borges Marques}
\address[Violeta Borges Marques]{Universiteit Antwerpen, Departement Wiskunde, Middelheimcampus,
Middelheimlaan 1,
2020 Antwerp, Belgium}
\email{Violeta.BorgesMarques@uantwerpen.be}

\author{Wendy Lowen} 
\address[Wendy Lowen]{Universiteit Antwerpen, Departement Wiskunde, Middelheimcampus,
Middelheimlaan 1,
2020 Antwerp, Belgium}
\email{wendy.lowen@uantwerpen.be}

\author{Arne Mertens}
\address[Arne Mertens]{Universiteit Antwerpen, Departement Wiskunde, Middelheimcampus,
Middelheimlaan 1,
2020 Antwerp, Belgium}
\email{arne.mertens@uantwerpen.be}

\thanks{
This project has received funding from the European Research Council (ERC) under the European Union’s Horizon 2020 research and innovation programme (grant agreement No. 817762).
}

\makeatletter
\@namedef{subjclassname@2020}{
  \textup{2022} Mathematics Subject Classification}
\makeatother

\subjclass[2022]{13D10, 18N60 (Primary), 18M05 (Secondary)}

\keywords{}

\usepackage{graphicx}
\usepackage[english]{babel}
\usepackage{amsmath}
\usepackage{amssymb}
\usepackage{amsthm}
\usepackage[all,cmtip]{xy}
\usepackage{cancel}
\usepackage[alpine]{ifsym}
\usepackage{enumerate}
\usepackage{stmaryrd}
\usepackage{tikz}
\usepackage{tikz-cd}
\usetikzlibrary{matrix}
\usepackage[toc,page]{appendix}
\usepackage{hyperref}
\usepackage{mathtools}
\usepackage[style=alphabetic, maxnames = 50, maxalphanames=50, urldate=long, firstinits=true]{biblatex}
\DeclareMathOperator{\Hom}{Hom}
\DeclareMathOperator{\id}{id}

\DeclareMathOperator{\Fun}{Fun}

\DeclareMathOperator*{\colim}{colim}

\DeclareMathOperator{\Set}{Set}

\DeclareMathOperator{\Mod}{Mod}

\DeclareMathOperator{\SSet}{SSet}

\DeclareMathOperator{\Colax}{Colax}

\DeclareMathOperator{\Quiv}{Quiv}
\DeclareMathOperator{\Cat}{Cat}

\newcommand{\fint}{\mathbf{\Delta}_{f}} 
\newcommand{\simp}{\mathbf{\Delta}} 
\newcommand{\nec}{\mathcal{N}ec} 
\newcommand{\ts}{S_{\otimes}} 

\newcommand{\lra}{\longrightarrow} 

\newcommand{\N}{\mathbb{N}}

\newcommand{\vvv}{\mathcal{V}} 

\newcommand{\www}{\mathcal{W}} 

\newcommand{\zzz}{\mathcal{Z}} 

\newcommand{\aaa}{\mathcal{A}} 

\newcommand{\mmm}{\mathcal{M}} 

\newcommand{\ccc}{\mathcal{C}}

\newcommand{\op}{\mathrm{op}} 

\newcommand{\necop}{\mathcal{N}ec^{\op}}

\DeclareMathOperator{\Ext}{Ext}

\DeclareMathOperator{\Tor}{Tor}

\DeclareMathOperator{\can}{can}

\newtheorem{Thm}{Theorem}[section]
\newtheorem*{Thm*}{Theorem}
\newtheorem{Lem}[Thm]{Lemma}
\newtheorem{Prop}[Thm]{Proposition}

\theoremstyle{definition}
\newtheorem{Def}[Thm]{Definition}
\newtheorem{Ex}[Thm]{Example}

\theoremstyle{remark}
\newtheorem{Rem}[Thm]{Remark}

\begin{document}

\begin{abstract}
The framework of templicial objects was put forth in \cite{LM1} in order to develop higher categorical concepts in the presence of enrichment. In particular, quasi-categories in modules constitute a subclass of templicial modules which may be considered as a kind of ``weak dg-categories (concentrated in homologically positive degrees)'' according to \cite{LM2}. The main goal of the present paper is to initiate the deformation theory of templicial modules. In particular, we show that quasi-categories in modules are preserved under levelwise flat infinitesimal deformation. 
\end{abstract}

\maketitle

\tableofcontents

\section{Introduction}

Algebraic deformation theory dates back to the foundational work of Gerstenhaber on algebra deformations \cite{gerstenhaber1} \cite{gerstenhaber2}. Its core ideas were later applied to various other objects of algebro-geometric nature, like bialgebras \cite{gerstenhaber1990bialgebra} \cite{quantum} \cite{tetra} \cite{ginot2016deformation} and presheaves of algebras or, more generally, prestacks \cite{gerstenhaber1988cohomology} \cite{lowen2011hochschild}
\cite{DLpre} \cite{box}. Deformation theory plays a prominent role in noncommutative geometry. For instance, the deformation theory of algebras culminates in the topic of deformation quantisation \cite{Konts}, bialgebra deformations are naturally relevant to the theory of quantum groups, and deformations of structure sheaves yield noncommutative schemes in the sense of Van den Bergh, modelled by abelian categories \cite{artintatevandenbergh} \cite{VdBquadrics} \cite{lowen2006deformation} \cite{lowenprestacks} \cite{VdBdefquant}. 

Meanwhile, derived and more general triangulated categories have taken centre stage in (noncommutative) algebraic geometry and neighbouring fields. More precisely, to develop geometrical ideas on this level one typically uses dg-categories ``enhancing'' the triangulated ones \cite{bondalkapranovenhanced} \cite{keller2006differential} \cite{KaledinHdR} \cite{Kal2}.
Sometimes one prefers to use the more flexible (but also algebraically more involved) $A_{\infty}$-categories instead \cite{lefevre} \cite{keller2001introduction}. 
In the context of deformations, one encounters variants that are to some extent ``curved'' \cite{positselski2018weakly} \cite{dedeken2018filtered}, as is the case for Fukaya categories in general \cite{fukaya2003deformation} \cite{fukaya2009lagrangian}.
This curvature phenomenon takes place in cohomological degree 2, and can sometimes, but not always, be negotiated \cite{KL} \cite{LVdBcur}.

The homotopy (Morita) theory of dg-categories was developed by To\"en in \cite{toen2007homotopy} building on Tabuada's model category structure \cite{tabuada2005structure}. In spite of the huge success of this theory, it has some drawbacks related to it being a ``strict'' higher categorical model. For example, the monoidal stucture and the model structure on this category do not give rise to a monoidal model category. 
In recent years, several attempts were undertaken to build new models for ``weak'' dg-categories inspired by the variety of models for $(\infty,1)$-categories available in algebraic topology \cite{bergner2007model}\cite{hirschowitz1998descente}\cite{rezk2001model}\cite{joyal2002quasi}. For instance, and most recently, a model of ``dg-Segal spaces'' was put forth by Dimitriadis Bermejo \cite{bermejo2022dgSegal}. Earlier on, ``dg-Segal categories'' have been investigated by Bacard building on work by Simpson and Leinster \cite{bacard2010Segal} \cite{simpson2012homotopy} \cite{leinster2000homotopy}. Approaches to linearity and more general enrichment of a different flavour are due to Lurie \cite{lurie2016higher} and Gepner-Haugseng \cite{gepner2015enriched}.

Recently in \cite{LM1}, \cite{LM2}, templicial objects in a monoidal category were put forth as an appropriate generalisation of simplicial sets upon replacing $\Set$ by a not necessarily cartesian monoidal category $\vvv$. Although more involved in terms of algebraic structure, these templicial objects - comprised of $\vvv$-objects - are just as tangible as simplicial sets. In particular, for $\vvv = \Mod(k)$ with $k$ a commutative ring, they are amenable to algebraic deformation theory. In \cite{LM1}, the framework of templicial objects is used in order to define quasi-categories in $\vvv$ as a counterpart of the highly succesful quasi-categories ``in sets'' of Boardman-Vogt \cite{boardmanvogt}, Joyal \cite{joyal2002quasi} and Lurie \cite{lurie2009higher}. According to \cite{LM2}, quasi-categories in modules may be considered as a kind of ``weak dg-categories concentrated in homologically positive degrees''. Note that dg-categories in those degrees have (uncurved) $A_{\infty}$-categories as deformations according to their Hochschild complex.

The main goal of the present paper is to show that \emph{the class of quasi-categories in modules is stable under levelwise flat infinitesimal deformation of templicial modules} (Theorem \ref{thmmainq}). 
This result does not follow in a straightforward manner from the definition of a quasi-category $X$ in $\vvv$, which is expressed in terms of a familiar ``weak Kan'' horn lifting property for associated necklicial modules (Definition \ref{DefwK}). Instead, one resorts to ``wings'' instead of ``horns'', making use of the fact that in the specific case $\vvv = \Mod(k)$, quasi-categories in modules are equivalently characterised through lifting wings (Definition \ref{Defwings} and Proposition \ref{propwwK}).
In Theorem \ref{thmwingsmain}, using truncations of wings, we prove a key stability property for levelwise flat quasi-categories with respect to external tensor products. The proof of Theorem \ref{thmmainq} is obtained by combining this with the relevant stability property with respect to extensions (Proposition \ref{propextKan}).

Further, inspired by the well known preservation of projectivity for modules, we show that \emph{the class of deg-projective templicial modules is also stable under levelwise flat infinitesimal deformation} (Theorem \ref{propdegproj}). 
This class consists of templicial modules with well behaved submodules of degenerate, resp. nondegenerate, simplices. In particular, for deg-projective templicial modules a version of the Eilenberg-Zilber lemma holds \cite{LM1}.


In future work, the deformation theory of quasi-categories in modules will be further developed in various directions. 
Since, by the results in the present paper, the main properties of interest are preserved under deformation, one may focus on deforming templicial modules from a purely algebraic perspective. 
In work in progress, this is done by means of an associated Hochschild complex. 
Further, we will investigate the relationship between deformations of templicial modules on the one hand and those of dg-categories on the other hand. For this we will make use of an enhancement of the dg-nerve \cite{lurie2016higher} that lands in quasi-categories in modules, as constructed in \cite{LM2} and further investigated in \cite{mertens2023nerves}.

\section{Quasi-categories in a monoidal category}\label{parparenriched}

In this section, we recall the framework of templicial objects (see \S \ref{partemp}) in a suitable monoidal category $\vvv$ (see \S \ref{parmon}) from \cite{LM1}. With the eye on deformation theory, our main interest in this paper is in templicial modules, that is, letting $\vvv$ be the category of modules over a commutative ring. The case $\vvv = \Set$ recovers simplicial sets as templicial objects in $\Set$.
In \S \ref{parqcat}, we introduce our main concept of interest, which is that of a quasi-category in $\vvv$. 
This is a templicial object in $\vvv$ which satisfies a familiar horn lifting property. In order to express this property, one requires necklicial $\vvv$-objects (see \S \ref{parnec}). We will make crucial use of the similar property of lifting wings,  which is equivalent to lifting horns for necklicial modules (see \S \ref{parwings}). Finally, in \S \ref{pardegproj}, we introduce another property of interest for templicial modules, namely deg-projectivity from \cite{LM2}, which expresses the existence of good submodules of degenerate and non-degenerate simplices.

\subsection{Monoidal setup}\label{parmon}
For general background on monoidal categories, enriched categories and (colax) monoidal functors we refer to \cite{kelly2005basic}\cite{aguiar2010monoidal}.
Throughout, let $(\mathcal{V},\otimes,I)$ be a fixed bicomplete, symmetric monoidal closed category (see e.g. \cite{street1974elementary}). Up to natural isomorphism, there is a unique colimit preserving functor 
\begin{equation}\label{eqF}
F: \Set\longrightarrow \mathcal{V}
\end{equation}
such that $F(\{*\})\cong I$. This functor is strong monoidal and left adjoint to the forgetful functor $U = \mathcal{V}(I,-): \mathcal{V}\lra \Set$. 

Recall that a $\vvv$-enriched category $\ccc$ with object set $S$ consists of a collection $(\ccc(a,b))_{a,b\in S}$ with $\ccc(a,b)\in \mathcal{V}$ endowed with multiplications
$$
\mathcal{C}(a,b)\otimes \mathcal{C}(b,c)\longrightarrow \mathcal{C}(a,c)
$$
and units $I \lra \ccc(a,a)$ satisfying the usual associativity and unitality axioms.

The quasi-categories in $\vvv$ introduced in \cite{LM1} (see \S \ref{parqcat}) are defined using appropriate generalisations of simplicial sets (see \S \ref{partemp}). In order to combine the change from $\Set$ to $\vvv$ for hom-spaces with the categorical viewpoint which requires an underlying set of objects, we make use of enriched quivers.
More precisely, given a set $S$, we refer to a collection $Q = (Q(a,b))_{a,b\in S}$ with $Q(a,b)\in \mathcal{V}$ as a \emph{$\mathcal{V}$-enriched quiver} with $S$ its set of \emph{vertices}. A \emph{quiver morphism} $f: Q\rightarrow P$ is a collection $(f_{a,b})_{a,b\in S}$ of morphisms $f_{a,b}: Q(a,b)\rightarrow P(a,b)$ in $\mathcal{V}$. We thus obtain a category
$$
\mathcal{V}\Quiv_{S}
$$
of $\mathcal{V}$-enriched quivers with a fixed vertex set $S$. This category can be equipped with a monoidal structure $(\otimes_{S},I_{S})$ such that a monoid in $\mathcal{V}\Quiv_{S}$ corresponds precisely to a $\vvv$-enriched category with object set $S$. We refer to \cite{LM1} for more details.

For the entire paper, let us fix a unital commutative ring $k$. Our main case of interest is the category $\vvv = \Mod(k)$ of $k$-modules with the tensor product $\otimes_{k}$ over $k$ as monoidal structure. Though many of the proofs below can be easily extended to more general $\vvv$, we will occasionally restrict to $k$-modules for convenience's sake.

\subsection{Templicial objects}\label{partemp}

We briefly recall the framework of templicial objects in $\vvv$ from \cite{LM1}, referring to loc. cit. for a more detailed exposition and the basic results. Roughly speaking, templicial objects can be seen as an appropriate generalisation of simplicial sets upon replacing $\Set$ by a not necessarily cartesian monoidal category $\vvv$.

Let $\simp$ be the simplex category, which consists of the finite ordinals $[n] = \{0,...,n\}$ for $n\geq 0$ as objects with order preserving maps as morphisms. We will make use of the finite interval category $\fint$, which is the subcategory of $\simp$ with the same objects and containing those morphisms $f: [m]\rightarrow [n]$ that preserve the endpoints, that is, $f(0) = 0$ and $f(m) = n$. The category $\fint$ is strict monoidal for $[n] + [m] =[n+m]$.

\begin{Def}\label{definition: temp. obj.}
A \emph{templicial object} in $\mathcal{V}$ is a pair $(X,S)$ with $S$ a set and
$$
X: \fint^{op}\longrightarrow \mathcal{V}\Quiv_{S}
$$
a strongly unital, colax monoidal functor. 
\end{Def}

For our purposes, we will mainly be interested in the category 
$$
\ts\mathcal{V}_S
$$
of templicial objects in $\vvv$ with fixed vertex set $S$, with monoidal natural transformation as morphisms.
We refer to loc. cit. for the definition of the category $\ts\mathcal{V}$ of templicial objects in $\vvv$ with varying object sets.

For $\vvv = \Set$, there is a canonical equivalence of categories $\SSet \cong \ts \vvv$ where $\SSet = \Fun(\simp^{op}, \Set)$ is the category of simplicial sets \cite[Proposition 2.7]{LM1}\cite[Proposition 3.1.7]{leinster2000homotopy}. Consequently, the functor $F: \Set \lra \vvv$ from \eqref{eqF} induces a functor
\begin{equation}\label{eqtildeF}
	\tilde{F}: \SSet \lra \ts \vvv.
\end{equation}
We will refer to templicial objects of the form $\tilde{F}(X)$ for $X \in \SSet$ as \emph{free} templicial objects in $\vvv$.

Let $(X,S)$ be a templicial object. The structure of $X$ as a functor $\fint^{op}\rightarrow \mathcal{V}\Quiv_{S}$ is equivalent to a collection of $\mathcal{V}$-quivers $(X_{n})_{n\geq 0}$ as well as quiver morphisms
\begin{align*}
d_{j} 
: X_{n}\longrightarrow X_{n-1}\quad \text{for all integers }0 < j < n\\
s_{i} 
: X_{n}\longrightarrow X_{n+1}\quad \text{for all integers }0\leq i\leq n
\end{align*}
called the \emph{inner face morphisms} and \emph{degeneracy morphisms} respectively, which satisfy the usual simplicial identities. The colax monoidal structure of $X$ provides it with \emph{comultiplications}, which are quiver morphisms
$$
\mu_{k,l}: X_{k+l}\longrightarrow X_{k}\otimes_{S} X_{l}
$$
for all $k,l\geq 0$ satisfying coassociativity and naturality conditions. Further, $X$ has a \emph{counit} $\epsilon: X_{0}\xlongrightarrow{\cong} I_{S}$, which is assumed to be an isomorphism by the strong unitality.

Note that, contrary to a simplicial object, $X$ does not have any outer face morphisms $d_{0}, d_{n}: X_{n}\longrightarrow X_{n-1}$. Instead, $X$ is equipped with the comultiplications $\mu_{k,l}$ which serve as a replacement for the outer faces in the (non-cartesian) monoidal context.

We end this section with the main motivating example for using templicial objects in $\vvv$ in order to model higher enriched categories.

\begin{Ex}\cite[\S 2.2]{LM1}\label{exnerve}
Let $\ccc$ be a $\vvv$-enriched category with object set $S$. The \emph{templicial nerve} of $\ccc$ is the templicial object $N_{\vvv}(\ccc)$ in $\vvv$ with
	\begin{equation}
		N_{\mathcal{V}}(\ccc)_n = \ccc^{\otimes_S n},
	\end{equation}
	with inner face maps induced by the multipliations of $\ccc$, degeneracies induced by the units of $\ccc$, and comultiplications given by the canonical isomorphisms $\ccc^{\otimes_S k+l} \cong \ccc^{\otimes_S k} \otimes_S \ccc^{\otimes_S l}$.
\end{Ex}

\subsection{Necklicial objects}\label{parnec}
 Let $\SSet_{*,*} = \SSet_{\partial\Delta^{1}/}$ denote the category of bipointed simplicial sets. The category $\nec$ is the full subcategory of $\SSet_{*,*}$ spanned by all \emph{necklaces}, that is sequences $\Delta^{n_{1}}\vee ...\vee \Delta^{n_{k}}$ of standard simplices, called \emph{beads}, which are glued at their endpoints (where $\vee$ denotes the wedge sum). Necklaces first appeared in \cite{baues1980geometry} (under a different name) and were later employed and popularised in \cite{dugger2011rigidification}. In \cite[Proposition 3.4]{LM1}, it was shown that $\nec$ admits the following combinatorial description:
\begin{itemize}
\item The objects of $\nec$ are all pairs $(T,p)$ with $p\geq 0$ an integer and $T\subseteq [p]$ a subset containing $\{0 < p\}$.
\item A morphism $(T,p)\rightarrow (U,q)$ in $\nec$ is a morphism $f: [p]\rightarrow [q]$ in $\fint$ such that $U\subseteq f(T)$.
\end{itemize}
Here, a subset $T = \{0 = t_{0} < t_{1} < ... < t_{k} = p\}$ of $[p]$ corresponds to the necklace $\Delta^{t_{1}}\vee \Delta^{t_{2}-t_{1}}\vee ...\vee \Delta^{p-t_{k-1}}$. For example, $T = \{0 < p\}$ is the single simplex $\Delta^{p}$ while $T = [p]$ is a sequence of edges. Then the wedge sum $\vee$ is given as follows:
$$
(T,p)\vee (U,q) = (T\cup (p + U), p+q)
$$
for any necklaces $(T,p)$ and $(U,q)$. This makes $\nec$ into a monoidal category with monoidal unit given by $(\{0\},0)$.

It will be useful to consider some particular classes of necklace maps.
We call a necklace map $f: (T,p)\rightarrow (U,q)$ \emph{inert} if the underlying morphism $f: [p]\rightarrow [q]$ in $\fint$ is the identity, and we call $f$ \emph{active} if $f(T) = U$. Every necklace map can be uniquely decomposed as an active map followed by an inert map.
Further, we call $f$ \emph{injective} if the underlying morphism $f: [p]\rightarrow [q]$ in $\fint$ is injective. 

Necklace maps allow us to treat the inner face maps and degeneracy maps of a templicial object $(X,S)$ on the one hand, and its comultiplications on the other hand, on the same footing. More concretely, given any necklace $T = \{0 = t_{0} < t_{1} < ... < t_{k} = p\}$, we set
$$
X_{T} = X_{t_{1}}\otimes_{S} X_{t_{2}-t_{1}}\otimes_{S} ...\otimes_{S} X_{p-t_{k-1}}\in \mathcal{V}\Quiv_{S}.
$$
This definition can be extended to a functor $X_{\bullet}: \nec^{op}\rightarrow \mathcal{V}\Quiv_{S}$ in which the active necklace maps parameterise the inner face maps and degeneracy maps of $X$ and the inert necklace maps parameterise its comultiplications (see \cite[Construction 3.9]{LM1}).

In particular, we can evaluate in any $a,b\in S$ to obtain a functor
\begin{equation}\label{diagram: hom-object of necklace cat. assoc. to temp. obj.}
X_{\bullet}(a,b): \nec^{op}\longrightarrow \mathcal{V}.
\end{equation}
In general, a functor $Y: \nec^{op}\rightarrow \mathcal{V}$ will be called a \emph{necklicial $\vvv$-object}.
The necklicial $\vvv$-objects $X_{\bullet}(a,b)$ associated to a templicial object $X$ in $\vvv$ are used in the definition of quasi-categories in $\vvv$, see \S \ref{parqcat}.

\begin{Rem}\label{remark: necklace cats.}
	The category $\Fun(\necop, \vvv)$ becomes monoidal for the Day convolution, and one may thus consider \emph{necklace categories}, that is categories enriched in $\Fun(\necop, \vvv)$. In \cite[Theorem 3.12]{LM1}, a fully faithful functor $(-)^{nec}: \ts\vvv \lra 
	\vvv \Cat_{\nec}$ is constructed, landing in the category $\vvv \Cat_{\nec}$ of small necklace categories. This functor is left adjoint to a functor 
	$(-)^{temp}:
	\vvv \Cat_{\nec} \lra \ts\vvv$.
\end{Rem}

\subsection{Quasi-categories in a monoidal category}\label{parqcat}

For the rest of this section, we assume that the forgetful functor $U = \mathcal{V}(I,-): \mathcal{V}\rightarrow \Set$ preserves and reflects regular epimorphisms. Note that this is the case for our main category of interest $\vvv = \Mod(k)$.

Given integers $0\leq j\leq n$, we denote by $\Delta^{n}$ and $\Lambda^{n}_{j}$ the standard $n$-simplex and the $j$th $n$-horn respectively. 

\begin{Def}\label{DefwK}
Let $Y: \nec^{op}\longrightarrow \mathcal{V}$ be a necklicial $\vvv$-object. We say that $Y$ is \emph{weak Kan} if for all $0 < j < n$ any lifting problem
\[\begin{tikzcd}
	{\tilde{F}(\Lambda^{n}_{j})_{\bullet}(0,n)} & Y \\
	{\tilde{F}(\Delta^{n})_{\bullet}(0,n)}
	\arrow[from=1-1, to=2-1]
	\arrow[from=1-1, to=1-2]
	\arrow[dashed, from=2-1, to=1-2]
\end{tikzcd}\]
where the vertical morphism is induced by the inclusion $\Lambda^{n}_{j}\subseteq \Delta^{n}$, has a solution in $\Fun(\necop, \mathcal{V})$. Here, $\tilde{F}$ is the free functor from \eqref{eqtildeF}. We call a templicial object $(X,S)$ in $\mathcal{V}$ a \emph{quasi-category in $\mathcal{V}$} if the functors $X_{\bullet}(a,b)$ are weak Kan for all $a,b\in S$. In this case, we refer to the elements of $S$ as the \emph{objects} of $X$ and to elements of $U(X_{1}(a,b))$ as \emph{morphisms} $a\rightarrow b$ in $X$.
\end{Def}

Let $[-,-]$ denote the canonical enrichment of $\Fun(\nec^{op},\mathcal{V})$ over $\mathcal{V}$. Then given a necklicial $\vvv$-object $Y$ and integers $0 < j < n$, we  write
$$
\Lambda^{j}_{n}Y = \left[\tilde{F}(\Lambda^{n}_{j})_{\bullet}(0,n),Y\right].
$$
It was observed in \cite[Proposition 5.1]{LM1} that
\begin{equation}\label{equation: horn as colimit}
\begin{aligned}
 (\Lambda^{n}_{j})_{\bullet}(0,n) &= \bigcup_{\substack{i=1\\ i\neq j}}^{n-1}\delta_{i}(\Delta^{n-1}_{\bullet}(0,n-1))\cup \bigcup_{k=1}^{n-1}(\Delta^{k}\vee \Delta^{n-k})_{\bullet}(0,n)\\
&\cong \colim_{\substack{f: (T,p) \hookrightarrow \Delta^n\\ f\neq \delta_{j} ,\id}}\!\!\!T_{\bullet}(0,p)
\end{aligned}
\end{equation}

as a subfunctor of $\Delta^{n}_{\bullet}(0,n)$. The colimit here is taken over the full subcategory of $(\nec_{/\Delta^n})^{op}$ spanned by all injective necklace maps $T\hookrightarrow \Delta^n$ except $\delta_{j}: \Delta^{n-1}\hookrightarrow \Delta^n$ and the identity on $\Delta^n$. It follows that
\begin{equation}\label{equation: horn limit}
\Lambda^{j}_{n}Y\cong \lim_{\substack{f: (T,p) \hookrightarrow \Delta^n\\ f\neq \delta_{j} ,\id}}\!\!\left[\tilde{F}(T)_{\bullet}(0,p),Y\right]\cong \lim_{\substack{f: (T,p) \hookrightarrow \Delta^n\\ f\neq \delta_{j} ,\id}}\!\!\!Y_{T}
\end{equation}
The following characterisation of the lifting property from Definition \ref{DefwK} was formulated for templicial objects in \cite[\S 4.1]{LM2}.

\begin{Prop}\label{propcharq}
Let $Y: \nec^{op}\rightarrow \mathcal{V}$ be a necklicial $\vvv$-object. The following are equivalent:
\begin{enumerate}
\item $Y$ is weak Kan;
\item the canonical morphism $Y_{n}\rightarrow \Lambda^{j}_{n}Y$ is a regular epimorphism for all $0 < j < n$.
\end{enumerate}
\end{Prop}

The following examples of quasi-categories in $\vvv$ can be found in \cite[Corollary 5.11]{LM1} and \cite[Corollary 4.15]{LM2} respectively.

\begin{Prop}
The following are quasi-categories in $\vvv$:
\begin{enumerate}[1.]
	\item the nerve $N_{\mathcal{V}}(\ccc)$ of a $\vvv$-enriched category $\ccc$ (see Example \ref{exnerve});
	\item the free templicial object $\tilde{F}(X)$ in $\mathcal{V} = \Mod(k)$ for an ordinary quasi-category $X \in \SSet$.
\end{enumerate}
\end{Prop}

\subsection{Wings}\label{parwings}
We will also make use of a lifting property with respect to the \emph{wings} $W^{n}$ of a simplex $\Delta^{n}$ for $n\geq 2$, which are defined as the union of its two outer faces \cite{LM2}. Given a necklace $(T,n)$ and $0 < j < n$, the unique inert necklace map $T\hookrightarrow \{0 < n\}$ can thus be identified with a composite of inclusions of bipointed simplicial sets:
$$
T\subseteq W^{n}\subseteq \Lambda^{n}_{j}\subseteq \Delta^{n}.
$$

\begin{Def}\label{Defwings}
For $n\geq 2$, we write $W^{n}$ for the simplicial subset of $\Delta^{n}$ defined by
$$
W^{n}([m]) = \{f: [m]\rightarrow [n]\mid f(m) < n\text{ or }f(0) > 1\}\,\subseteq \, \Delta^{n}([m])
$$
for all $m\geq 0$. We call $W^{n}$ the \emph{wings} of $\Delta^{n}$. We say a functor $Y: \nec^{op}\rightarrow \mathcal{V}$ \emph{lifts wings} if for all $n\geq 2$, any lifting problem in $\Fun(\necop, \mathcal{V})$:
\[\begin{tikzcd}
	{\tilde{F}(W^{n})_{\bullet}(0,n)} & {Y} \\
	{\tilde{F}(\Delta^{n})_{\bullet}(0,n)}
	\arrow[from=1-1, to=1-2]
	\arrow[from=1-1, to=2-1]
	\arrow[dashed, from=2-1, to=1-2]
\end{tikzcd}\]
where the vertical morphism is induced by the inclusion $W^{n}\subseteq \Delta^{n}$, has a solution. We say that a templicial object $(X,S)$ in $\mathcal{V}$ \emph{lifts wings} if the functors $X_{\bullet}(a,b)$ lift wings for all $a,b\in S$.
\end{Def}

The following result was formulated for quasi-categories in $\vvv$ in \cite{LM2}.

\begin{Prop}\label{propwwK}
	Let $Y: \necop \lra \vvv$ be a necklicial $\vvv$-object. 
	\begin{enumerate}[1.]
	\item If $Y$ is weak Kan, then $Y$ lifts wings.
	\item For $\vvv = \Mod(k)$, $Y$ is weak Kan if and only if $Y$ lifts wings.
	\end{enumerate}	
\end{Prop}
\begin{proof}
(1) immediately follows from \cite[Lemma 4.9]{LM2}. The proof of (2) is completely analogous to that of \cite[Prop. 4.12]{LM2}.
\end{proof}

From \cite[Proposition 4.8]{LM2}, we have for all $n\geq 2$: 
$$
W^{n}_{\bullet}(0,n) = \bigcup_{k=1}^{n-1}(\Delta^{k}\vee \Delta^{n-k})_{\bullet}(0,n)\cong \colim_{\substack{T\hookrightarrow \Delta^{n}\\ f\neq \id 
\text{ inert}}}T_{\bullet}(0,n)
$$
as a subfunctor of $\Delta^{n}_{\bullet}(0,n)$. Similarly to \eqref{equation: horn limit}, we put
$$
W_{n}Y = [\tilde{F}(W^{n})_{\bullet}(0,n),Y]\cong \lim_{\substack{T\hookrightarrow \Delta^{n}\\ f\neq \id 
\text{ inert}}}Y_{T}
$$
for any necklicial $\vvv$-object Y: $\nec^{op}\rightarrow \mathcal{V}$. Consequently, we have:

\begin{Prop}\label{propwingssurj}
Let $Y: \necop \lra \vvv$ be a necklicial $\mathcal{V}$-object. The following are equivalent.
\begin{enumerate}
\item $Y$ lifts wings;
\item the canonical morphism $Y_{n}\longrightarrow W_{n}Y$ is a regular epimorphism for all $n\geq 2$.
\end{enumerate}
\end{Prop}

\subsection{Deg-projective templicial modules}\label{pardegproj}

Let $(X,S)$ be a templicial $\vvv$-object. Following \cite{LM1}, we put for $n\geq 0$:
\begin{equation}
{{X}^{deg}_n} = \colim_{\substack{f:[n] \twoheadrightarrow [m] \\ f\neq \id }} X_m \in \vvv\Quiv_{S}
\end{equation}
where the colimit is taken over all non-identity surjective morphisms $[n] \rightarrow [m]$ in $\fint$. We can consider $X^{deg}_{n}$ as the quiver of degenerate $n$-simplices of $X$. Note that there is an associated canonical morphism $\can_n: X^{\deg}_n \lra X_n$.

\begin{Def}
A templicial object $(X,S)$ is called \emph{deg-projective} if for all $n\geq 0$, $\can_n$ has the left lifting property with respect to the class of regular epimorphisms.
\end{Def}

\begin{Ex}
Every simplicial set $X$ is deg-projective in $\SSet$ (considered as a templicial set), and $\tilde{F}(X)$ is deg-projective in $\ts\vvv$.
\end{Ex}

Since we are mainly interested in templicial modules, let us conveniently reformulate deg-projectivity in this context.

\begin{Ex}
Let $(X,S)$ be a templicial $k$-module. Note that $X$ is deg-projective if and only if $\can_{n}$ is a monomorphism with projective cokernel for all $n\geq 0$. Now consider the canonical exact sequence induced by $\can_n$:
\begin{equation}\label{eqcan0}
	\xymatrix{{E_X:} & {{X}^{deg}_n} \ar[r]^{{\can}_n} & {{X}_n} \ar[r]_{} & {{X}^{nd}_n} \ar[r] & 0
	}
\end{equation}
Then $X$ is deg-projective precisely if $E_X$ is split short exact for all $n\geq 0$.

In this case $X^{deg}_{n}$ really is the \emph{sub}quiver of $X_n$ of degenerate $n$-simplices. Similarly, we can consider $X^{nd}_{n}$ as the subquiver of non-degenerate $n$-simplices. In fact, by \cite[Lemma 2.19]{LM1} we have a version of the Eilenberg-Zilber lemma, that is,
\begin{equation}\label{eqEZ}
X_n\cong \bigoplus_{[n] \twoheadrightarrow [m]} X_m^{nd}
\end{equation}
where the direct sum is taken over all surjective morphisms $[n] \rightarrow [m]$ in $\fint$.
\end{Ex}

Let us end the section by making the distinction between deg-projectivity and levelwise projectivity for templicial modules.

\begin{Def}\label{definition: levelwise flat temp. obj.}
A templicial $k$-module $(X,S)$ is called \emph{levelwise projective} (resp. \emph{levelwise flat}) if $X_n(a,b)$ is a flat (resp. projective) $k$-module for all $n\geq 0$ and $a,b\in S$.
\end{Def}

Clearly any levelwise projective templicial $k$-module is levelwise flat.


\begin{Prop}\label{propdeglevel}
Any deg-projective templicial $k$-module $X$ is levelwise projective and $X^{deg}_{n}(a,b)$ is a projective $k$-module for all $n\geq 0$ and $a,b\in S$.
\end{Prop}
\begin{proof}
Suppose $X$ is deg-projective. By \eqref{eqEZ}, $X_n$ is projective and since $X_n \cong X^{deg}_n \oplus X^{nd}_n$, the direct summand $X^{deg}_n$ is projective too.
\end{proof}

\begin{Ex}
Certainly not every levelwise projective templicial $k$-module is deg-projective. Let $s_{0}: \mathbb{Z}\rightarrow \mathbb{Z}$ be the map given by multiplying by $2$. It can be extended to a templicial $\mathbb{Z}$-module with a single vertex $*$ by setting $X_{n}(*,*) = \mathbb{Z}$ for all $n\geq 0$ and setting all other degeneracy, inner face and comultiplication maps to be the identity on $\mathbb{Z}$. Then $X$ is clearly levelwise projective but $X^{nd}_{1}(*,*)\cong \mathbb{Z}/2\mathbb{Z}$ is not projective.
\end{Ex}

\section{Stability properties of necklicial modules}\label{parparstab}

In this section, we collect the main stability properties of necklicial modules that we will use in \S \ref{parpardefqcat}. We are mainly interested in the necklicial modules $X_{\bullet}(a,b)$ for objects $a,b \in S$ associated to a quasi-category $X$ in modules, and the extent to which the weak Kan property of $X_{\bullet}(a,b)$ is preserved.
First, we look into stability under the application of functors. In the generality of necklicial $\vvv$-objects, as the weak Kan property is expressed in terms of finite limits and colimits, we merely obtain preservation by exact functors (see \S \ref{parbasenec}). Specifying to a quasi-category $X$ in modules, in \S \ref{parqtens} we prove our key technical result making use of wings rather than horns. Precisely, in Theorem \ref{thmwingsmain} we show that if $X$ is levelwise flat over $k$, then for any $k$-module $N$ the external tensor products $X_{\bullet}(a,b) \otimes_k N$ are weak Kan.
Finally, in \S \ref{parextnec}, we show that the weak Kan property of necklicial $k$-modules is stable under extensions in $\Fun(\necop, \Mod(k))$.

\subsection{Base change for necklicial objects}\label{parbasenec}

In this section $\vvv$ and $\www$ are finitely bicomplete categories, but not necessarily monoidal. Note that because of \eqref{equation: horn as colimit}, the lifting condition of Definition \ref{DefwK} still makes sense in this context. Further recall that a functor is called \emph{exact} if it preserves finite limits and colimits. 

\begin{Prop}\label{propFpres}
Let $H:\vvv \lra \www$ be a functor and consider the induced
\begin{equation}
	H\circ -: \Fun(\necop, \vvv) \lra \Fun(\necop, \www).
\end{equation}
If $H$ is exact, then $H\circ -$ preserves weak Kan necklicial objects.
\end{Prop}

\begin{proof}
This follows from the fact that in Proposition \ref{propcharq} the property of being weak Kan is expressed in terms of finite limits and colimits.
\end{proof}

\subsection{Quasi-categories in modules under external tensoring}\label{parqtens}

In this section, we prove our key technical result, namely that for the necklicial $k$-modules associated to a levelwise flat templicial $k$-module, the weak Kan property is stable under tensoring externally by an arbitrary $k$-module. Recall that the tensor product of quivers is denoted by $\otimes_S$. In the sequel, we also consider the external tensor product $\otimes = \otimes_k$ over $k$ of a necklicial $k$-module $Y: \necop \lra \Mod(k)$ with a $k$-module $N$, given by $(Y \otimes N)_T = Y_T \otimes_k N$.

We make fundamental use of the wings from \S \ref{parwings}, or more specifically a truncated version thereof. Given integers $0\leq i < n$, let us denote
$$
W^{n}_{\leq i} = \bigcup_{k=1}^{i}(\Delta^{k}\vee \Delta^{n-k})_{\bullet}(0,n)\subseteq \Delta^{n}_{\bullet}(0,n).
$$
Note that $W^{n}_{\leq 0} = \emptyset$ and $W^{n}_{\leq n-1} = W^{n}_{\bullet}(0,n)$. Similarly to \eqref{equation: horn limit}, we set
$$
W_{n}^{\leq i}Y = [F(W^{n}_{\leq i}),Y]\cong \lim_{\substack{f:T {\hookrightarrow} \Delta^{n} \\ f\neq \id \text{ inert}\\ T\cap \{i+1,...,n-1\} = \emptyset}}Y_{T},
$$
for every necklicial $k$-module $Y: \nec^{op}\rightarrow \Mod(k)$. Further, for a functor $Y: \necop \longrightarrow k\Quiv_{S}$, we will also write $W_{n}^{\leq i}Y$ for the quiver given by $(W_{n}^{\leq i}Y)(a,b) = W_{n}^{\leq i}(Y(a,b))$ for all $a,b\in S$, where $Y(a,b)$ is the induced functor $\necop \longrightarrow \Mod(k)$ . We will mainly use this notation for $Y = X_{\bullet}\otimes_k N$, with $X$ a templicial object and $N$ a $k$-module.

\begin{Lem}\label{lemma: pushout of wedges}
Let $0 < i < n$ be integers. The following square is a pushout in $\Fun(\necop, \Set)$:
$$
\xymatrix{
\bigcup^{i-1}_{k=1}(\Delta^{k}\vee \Delta^{i-k}\vee \Delta^{n-i})_{\bullet}(0,n)\ar[r]\ar[d] & (\Delta^{i}\vee \Delta^{n-i})_{\bullet}(0,n)\ar[d]\\
W^{n}_{\leq i-1}\ar[r] & W^{n}_{\leq i}
}
$$
\end{Lem}
\begin{proof}
As subfunctors of $\Delta^{n}_{\bullet}(0,n)$, it suffices to note that $\bigcup^{i-1}_{k=1}(\Delta^{k}\vee \Delta^{i-k}\vee \Delta^{n-i})_{\bullet}(0,n)$ is precisely the intersection of $W^{n}_{\leq i-1}$ and $(\Delta^{i}\vee \Delta^{n-i})_{\bullet}(0,n)$, which is a straight forward verification.
\end{proof}

\begin{Lem}\label{lemma: pullback of wedges}
Let $(X,S)$ be a levelwise flat templicial $k$-module and $N$ a $k$-module. Then for all integers $0 < i < n$, the following square is a pullback in $k\Quiv_{S}$:
$$
\xymatrix{
W_{n}^{\leq i}(X_{\bullet}\otimes N)\ar[r]\ar[d] & (X_{i}\otimes_{S} X_{n-i})\otimes N\ar[d]\\
W_{n}^{\leq i-1}(X_{\bullet}\otimes N)\ar[r] & W_{i}(X_{\bullet}\otimes N)\otimes_{S} X_{n-i}
}
$$
\end{Lem}
\begin{proof}
Apply the functor $[-,X_{\bullet}\otimes N]: \Fun(\necop, \Set) \rightarrow k\Quiv_{S}$ to the pushout of Lemma \ref{lemma: pushout of wedges} to obtain a pullback:
$$
\xymatrix{
W_{n}^{\leq i}(X_{\bullet}\otimes N)\ar[r]\ar[d] & (X_{i}\otimes_{S} X_{n-i})\otimes N\ar[d]\\
W_{n}^{\leq i-1}(X_{\bullet}\otimes N)\ar[r] & \underset{\substack{f:T {\hookrightarrow} \Delta^{i}\vee \Delta^{n-1} \\ f\neq \id \text{ inert}\\ T\cap \{i+1,...,n-1\} = \emptyset}}{\lim}\left(X_{T}\otimes N\right)
}
$$

Then note that we have isomorphisms of $k$-modules for all $a,b\in S$:
\begin{align*}
\lim_{\substack{f:T{\hookrightarrow} \Delta^{i}\vee \Delta^{n-i}\\ f\neq \id \text{ inert}\\ T\cap \{i+1,...,n-1\} = \emptyset}}\left(X_{T}(a,b)\otimes N\right) &\cong \lim_{\substack{f:U{\hookrightarrow} \Delta^{i}\\ f\neq \id \text{ inert}}}\left(\bigoplus_{c\in S}X_{U}(a,c)\otimes X_{n-i}(c,b)\otimes N\right)\\
&\cong \bigoplus_{c\in S}\lim_{\substack{f:U{\hookrightarrow} \Delta^{i}\\ f\neq \id \text{ inert}}}\left(X_{U}(a,c)\otimes N\right)\otimes X_{n-i}(c,b)
\end{align*}
The first isomorphism holds because every inert map $T\hookrightarrow \Delta^{i}\vee \Delta^{n-i}$ with $T\cap \{i+1,...,n-1\} = \emptyset$ can be written as $f\vee \id_{\Delta^{n-i}}$ for a unique inert map $f: U\hookrightarrow \Delta^{i}$. The second isomorphism holds because coproducts commute with finite limits in $\Mod(k)$, and by the flatness of $X_{n-i}(c,b)$. Therefore, we find that the bottom right quiver in the pullback above is isomorphic to $W_{i}(X_{\bullet}\otimes N)\otimes_{S} X_{n-i}$ as desired.
\end{proof}

\begin{Lem}\label{lemma: flat pullback}
\par Consider a pullback diagram of $k$-modules
\[
\begin{tikzcd}
P \arrow{r}\arrow{d} & A\arrow{d}{\mu}\\
B \arrow{r}{f} & C\\
\end{tikzcd}
\]
\noindent where $\mu$ is an epimorphism and $A$, $B$ and $C$ are flat. Then $P$ is flat and for all $k$-modules $N$,
\[
N\otimes P \cong N\otimes A \times_{N\otimes C} N\otimes B .
\]
\end{Lem}
\begin{proof}
\par This readily follows from considering the short exact sequence
\[
\begin{tikzcd}
        0 \arrow{r} & P \arrow{r}& A \oplus B \arrow{r}{\mu-f} & C\arrow{r} & 0 .
\end{tikzcd}
\]
\noindent where both $A\oplus B$ and $C$ are flat (see e.g. \cite[Tags 00HL and 00HM]{stacks-project}). 
\end{proof}

\begin{Prop}\label{proposition: flat wings preserve tensor}
Let $(X,S)$ be a levelwise flat templicial $k$-module that lifts wings. Then for all integers $0\leq i < n$ the quiver $W_{n}^{\leq i}X$ is flat and for all $k$-modules $N$, we have an isomorphism of quivers:
$$
W_{n}^{\leq i}(X_{\bullet}\otimes N)\cong W_{n}^{\leq i}X\otimes N.
$$
In particular, $W_{n}(X_{\bullet}\otimes N)\cong W_{n}X\otimes N$.
\end{Prop}
\begin{proof}
We prove the statement by induction on $(i,n)$. The case $i = 0$ is trivial since $W_{n}^{\leq 0}X = 0$. Now let $0 < i < n$. We assume that we have the statement for all $(j,n)$ with either $m < n$, or $m = n$ and $j < i$.

Consider the pullback from Lemma \ref{lemma: pullback of wedges} with $N = k$:
$$
\xymatrix{
W_{n}^{\leq i}X\ar[r]\ar[d] & X_{i}\otimes_{S} X_{n-i}\ar[d]\\
W_{n}^{\leq i-1}X\ar[r] & W_{i}X\otimes_{S} X_{n-i}
}
$$
Since $X$ is levelwise flat and lifts wings, it follows from the induction hypothesis and Lemma \ref{lemma: flat pullback} that $W_{n}^{\leq i}X$ is again flat. Moreover, we obtain a pullback diagram:
$$
\xymatrix{
W_{n}^{\leq i}X\otimes N\ar[r]\ar[d] & \left(X_{i}\otimes_{S} X_{n-i}\right)\otimes N\ar[d]\\
W_{n}^{\leq i-1}X\otimes N\ar[r] & \left(W_{i}X\otimes_{S} X_{n-i}\right)\otimes N
}
$$
It follows from the induction hypothesis that we have isomorphisms of quivers
$$
W_{n}^{\leq i-1}X\otimes N\cong W_{n}^{\leq i}(X_{\bullet}\otimes N) \text{ and } (W_{i}X\otimes_{S} X_{n-i})\otimes N\cong W_{i}(X_{\bullet}\otimes N)\otimes_{S} X_{n-i}
$$
and thus the above pullback is precisely the pullback from Lemma \ref{lemma: pullback of wedges} (for arbitrary $N$). Hence, we obtain $W_{n}^{\leq i}X\otimes N\cong W_{n}^{\leq i}(X_{\bullet}\otimes N)$.
\end{proof}

\begin{Thm}\label{thmwingsmain}
Let $(X,S)$ be a levelwise flat quasi-category in $\Mod(k)$. For any $k$-module $N$ and $a,b\in S$, $X_{\bullet}(a,b)\otimes_k N$ is a weak Kan necklicial $k$-module.
\end{Thm}
\begin{proof}
Let $n\geq 2$. In view of Propositions \ref{propwwK} and \ref{propwingssurj}, we have that $X_{n}(a,b)\longrightarrow W_{n}(X)(a,b)$ is surjective and it suffices to show that $$X_{n}(a,b)\otimes N\longrightarrow W_{n}(X_{\bullet}(a,b)\otimes N)$$ is surjective as well. The latter is now clear from Proposition \ref{proposition: flat wings preserve tensor}.
\end{proof}

\subsection{Extensions of weak Kan necklicial objects}\label{parextnec}

The main goal of this section is to show that for $\vvv$ a (not necessarily monoidal) abelian category, the class of weak Kan necklicial $\vvv$-objects is closed under extensions in the category $\Fun(\necop, \vvv)$.
We start with a more general result formulated in an arbitrary abelian category $\aaa$.

\begin{Lem}\label{lemZ}
Let $\aaa$ be an abelian category. Let $\mmm$ be a class of monomorphisms with projective codomains in $\aaa$. Let $\zzz$ denote the class of the cokernels of all morphisms in $\mmm$ and let $\www$ denote the class of objects with the right lifting property with respect to $\mmm$.
Consider an object $X \in \aaa$. The following are equivalent:
\begin{enumerate}
\item $X \in \www$;
\item $\Ext^1(Z,X) = 0, \,\,\, \forall Z \in \zzz$.
\end{enumerate}	
\end{Lem}

\begin{proof}
	For every $f: A \lra P$ in $\mmm$, consider the short exact sequence
	$0 \lra A \lra P \lra Z \lra 0$ with $Z \in \zzz$. The associated long exact sequence for $\Ext^{\ast}(-,X)$ features 
	$$\Hom(P,X) \lra \Hom(A,X) \lra \Ext^1(Z,X) \lra \Ext^1(P,X).$$
	By definition we have $X \in \www$ if and only if the first map is always surjective.
	By projectivity of $P$, we have $\Ext^1(P,X) = 0$ and so $X \in \www$ if and only if $\Ext^1(Z,X) =0$.
\end{proof}
	
\begin{Prop}\label{propextab}
Let $\aaa$ be an abelian category. Let $\mmm$ be a class of monomorphisms with projective codomains in $\aaa$ and let $\www$ denote the class of objects with the right lifting property with respect to $\mmm$. Then $\www$ is closed under extensions.
	\end{Prop}
	
	\begin{proof}
 Immediate from Lemma \ref{lemZ}.
	\end{proof}

\begin{Prop}\label{propextKan}
	Let $\vvv$ be an abelian category. The class of weak Kan necklicial $\vvv$-objects is closed under extensions in $\Fun(\necop, \vvv)$.
\end{Prop}	

\begin{proof}
	This follows from Proposition \ref{propextab} for $\aaa = \Fun(\necop, \vvv)$ by taking $\mmm$ to consist of the images of the inner horn inclusions $(\Lambda^n_j)_{\bullet}(0,n) \lra \Delta^n_{\bullet}(0,n)$ for all $0<j<n$ under the functor 
	$F\circ -: \Fun(\necop, \Set) \lra \Fun(\necop, \vvv)$ induced by $F: \Set \lra \vvv$ from \eqref{eqF}. For this choice, $\www$ is precisely the class of weak Kan necklicial $\vvv$-objects.
\end{proof}

\section{Deformations of quasi-categories in modules}\label{parpardefnec}
\label{parpardefqcat}

In this section, we consider levelwise flat infinitesimal deformations of templicial modules (see \S \ref{pardefsetup} for our precise deformation setup, which encompasses deformations in the direction of Artinian local $k$-algebras over a field $k$).
In \S \ref{pardefqcat}, making use of the stability properties from \S \ref{parparstab}, we prove that the property of being a quasi-category in modules is preserved under deformation (Theorem \ref{thmmainq}). Finally, in \S \ref{pardefdegproj}, inspired by the well known preservation of projectivity for modules, we show that deg-projectivity of templicial modules is preserved under deformation (Theorem \ref{propdegproj}).

It is possible to deform templicial objects in suitable monoidal abelian categories rather than module categories, making use of the deformation theory of abelian categories from \cite{lowen2006deformation}. Details will appear elsewhere.

\subsection{Deformation setup} \label{pardefsetup}
Let $\theta: R \lra k$ be a ring homomorphism for a commutative ring $R$. 
This gives rise to a forgetful functor $\Mod(k) \lra \Mod(R)$
with left adjoint $k \otimes_R -: \Mod(R) \lra \Mod(k)$.
In order to develop infinitesimal deformation theory along $\theta$, we make some further assumptions. 

Firstly, we assume $\theta$ to be surjective. This in particular ensures that the functor $\Mod(k) \lra \Mod(R)$ is fully faithful. We will usually refrain from making the forgetful functor explicit in our notations, simply referring to ``$k$-modules considered as $R$-modules'' instead.

Secondly, we assume $I = \mathrm{Ker}(\theta)$ to be nilpotent, that is, $I^n = 0$ for some $n \geq 1$. 
When deforming $k$-modules into $R$-modules, it is customary to require flatness over the respective rings in order to control both the spaces of deformations and the properties of deformations. Note that by change of rings, a flat $R$-module $M$ gives rise to the flat $k$-module $k \otimes_R M$. For a mathematical object defined over a commutative ground ring and comprised of a collection of modules equipped with additional structure, one may likewise impose ``levelwise flatness'' - i.e. flatness of the individual modules involved - as a natural requirement for deformation theory (see Definition \ref{definition: levelwise flat temp. obj.}).

Further note that any surjective ring map $\theta: R \lra k$ with nilpotent kernel as above can be factored as a composition of ``small'' such maps (that is maps with square zero kernel) $\theta_i: R_{i+1} \lra R_{i}$ for $i = 0, \dots, n-2$ with $R_0 = k$ and $R_{n-1} = R$. Hence, a flat $R$-module $M$ gives rise to intermediate deformations $R_{i+1} \otimes_R M$ of $R_i \otimes_R M$ along $\theta_i$. As a consequence, to prove that a property is preserved by (levelwise) flat deformations, it suffices to consider the case where $I^2 =0$.

Note that our general setup applies in particular to classical infinitesimal deformations, where $R$ is an Artinian local $k$-algebra with maximal ideal $I$ and residue field $k = R/I$.

\subsection{Deformations of quasi-categories in modules} \label{pardefqcat}

In order to consider base change for templicial objects,
we have to be more restrictive in the functors we use compared to the situation for necklicial objects in \S \ref{parparstab}.  The reason is that templicial objects are colax monoidal functors, and should be composed with functors of the same type. 

We will be interested in deforming templicial modules with a fixed set of vertices. For a set $S$, we thus consider
$$
\ts\Mod(k)_S\cong \Colax_{su}(\fint^{\op}, k\Quiv_S).
$$
the category of strongly unital, colax monoidal functors $\fint^{\op}\rightarrow k\Quiv_S$ and monoidal natural transformations between them.

Let $\theta: R \lra k$ be a ring homomorphism as in \S \ref{pardefsetup}.
Since $k \otimes_R -: \Mod(R) \lra \Mod(k)$ is strong monoidal, there is an induced functor
\begin{equation}\label{eqtenscolax}
	k \otimes_R -: \ts\Mod(R)_S \lra \ts\Mod(k)_S
\end{equation}
with $(k \otimes_R X)_n(a,b) = k \otimes_R (X_n(a,b))$ for $n \in \N$ and $a,b \in S$.

\begin{Rem}
Note that the right adjoint of $k \otimes_R -: \Mod(R)\rightarrow \Mod(k)$ is lax monoidal but not colax monoidal, and does not give rise to a post-composition type functor  $\ts\Mod(k) \lra \ts\Mod(R)$. However, it does induce a forgetful functor $k\Cat_{\nec}\rightarrow R\Cat_{\nec}$ of necklace categories (see  Remark \ref{remark: necklace cats.}). Then $k \otimes_R -: \ts\Mod(R) \lra \ts\Mod(k)$ still has a right-adjoint by \cite{mertens2023nerves}, given by the composite $\ts\Mod(k)\xrightarrow{(-)^{nec}} k\Cat_{\nec}\rightarrow R\Cat_{\nec}\xrightarrow{(-)^{temp}} \ts\Mod(R)$.
\end{Rem}

\begin{Def}
Let $X \in \ts\Mod(k)_S$ be levelwise $k$-flat. An \emph{$R$-deformation of $X$} is a levelwise $R$-flat templicial $R$-module $\bar{X} \in \ts\Mod(R)_S$ together with an isomorphism $k \otimes_R \bar{X} \cong X$ in $\ts\Mod(k)_S$. Consider two $R$-deformations $\bar{X}$ and $\bar{X}'$ of $X$.  An \emph{equivalence} of $R$-deformations is an isomorphism of templicial $R$-modules $\Psi: \bar{X}\cong \bar{X}'$ such that $k\otimes_R \Psi = \id_X$.
\end{Def}

\begin{Ex}
Let $\bar{\ccc}$ be a (flat) $R$-deformation of a $k$-linear category $\ccc$ (see \cite{lowen2006deformation}). Then $N_R(\bar{\ccc})$ is an $R$-deformation of $N_k(\ccc)$ (see Example \ref{exnerve}).
\end{Ex}

\begin{Ex}
\label{ex:free}
Let $X$ be a simplicial set with $S = X_0$. Consider the free functor $\tilde{F}_k: \SSet_S \lra \ts\Mod(k)_S$ for any commutative ring $k$. 
Then with respect to the ring map $\theta: R \lra k$ we have 
$$\tilde{F}_k\cong (k \otimes_R -)\tilde{F}_R$$
for $k \otimes_R -$ from \eqref{eqtenscolax}. 
Further, $\tilde{F}_k(X)_n(a,b) = \oplus_{x\in X_n(a,b)}k$ is obviously $k$-flat so $\tilde{F}_R(X)$ is an $R$-deformation of $\tilde{F}_k(X)$.
\end{Ex}


\begin{Ex}
Let us give a concrete first order templicial deformation, putting $R = k[\epsilon]$ with $\epsilon^2 = 0$.
We define $P = \Tilde{F}(\Delta^2\coprod_{\Delta^1}\partial \Delta^2)$ using the inclusions $\delta_{1}: \Delta^1 \to \Delta^2$ and $\delta_{1}: \Delta^1 \to \partial \Delta^2$ in $\SSet$. This templicial module thus has the following non-degenerate simplices
\begin{align*}
    f_1\in P_1(a,b_1), &\quad g_1\in P_1(b_1,c)\\
    f_2\in P_1(a,b_2), &\quad g_2\in P_1(b_2,c)\\
    h\in P_1(a,c), & \quad \alpha \in P_2(a,c)
\end{align*}
\noindent with $\mu_{1,1}(\alpha)=f_1\otimes g_1$ and $d_1(\alpha)=h$. We describe a $k[\epsilon]$-deformation of $P$. On objects, $\bar{P}_n=k[\epsilon] \otimes_ k P_n\in k[\epsilon] \Quiv_{\{a,b_1,b_2,c\}}$. It is enough then to specify 
\begin{align*}
    \bar{\mu}_{1,1}(\alpha)&=f_1\otimes g_1 + f_2\otimes g_2\epsilon\\
    \bar{d}_1(\alpha)&=h
\end{align*}
\noindent Observe that, similarly to \cite[Example 2.10]{LM1}, the 2-simplex $\alpha\in\bar{P}_2(a,c)$ does not have a definable ``middle vertex'' as $\bar{\mu}_{1,1}(\alpha)$ is not a pure tensor. As such, we obtain a non-free deformation of a free templicial object, expanding on Example \ref{ex:free}. Below we show a pictorial representation of $P$ and $\bar{P}$, on the left and right, respectively. 

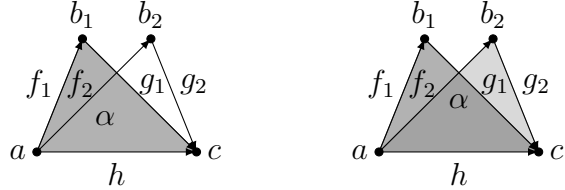
\begin{center}
\begin{tikzpicture}[scale=1.5]
\filldraw[fill=gray,opacity=0.6]
(-0.2,-0.3) -- (0.2,0.7) -- (1.2,-0.3);
\filldraw
(-0.2,-0.3) circle (1pt) node[left]{$a$}
(0.2,0.7) circle (1pt) node[above]{$b_{1}$}
(0.8,0.7) circle (1pt) node[above]{$b_{2}$}
(1.2,-0.3) circle (1pt) node[right]{$c$}
(0.4,0) node{$\alpha$};
\draw[-latex] (-0.2,-0.3) -- node[left,pos=0.6]{$f_{1}$} (0.2,0.7);
\draw[-latex] (-0.2,-0.3) -- node[left,pos=0.6]{$f_{2}$} (0.8,0.7);
\draw[-latex] (0.2,0.7) -- node[right,pos=0.4]{$g_{1}$} (1.2,-0.3);
\draw[-latex] (0.8,0.7) -- node[right,pos=0.4]{$g_{2}$} (1.2,-0.3);
\draw[-latex] (-0.2,-0.3) -- node[below,pos=0.5]{$h$} (1.2,-0.3);

\filldraw[fill=gray,opacity=0.6]
(2.8,-0.3) -- (3.2,0.7) -- (4.2,-0.3);
\filldraw[fill=gray,opacity=0.3]
(2.8,-0.3) -- (3.8,0.7) -- (4.2,-0.3);
\filldraw
(2.8,-0.3) circle (1pt) node[left]{$a$}
(3.2,0.7) circle (1pt) node[above]{$b_{1}$}
(3.8,0.7) circle (1pt) node[above]{$b_{2}$}
(4.2,-0.3) circle (1pt) node[right]{$c$}
(3.5,0.15) node{$\alpha$};
\draw[-latex] (2.8,-0.3) -- node[left,pos=0.6]{$f_{1}$} (3.2,0.7);
\draw[-latex] (2.8,-0.3) -- node[left,pos=0.6]{$f_{2}$} (3.8,0.7);
\draw[-latex] (3.2,0.7) -- node[right,pos=0.4]{$g_{1}$} (4.2,-0.3);
\draw[-latex] (3.8,0.7) -- node[right,pos=0.4]{$g_{2}$} (4.2,-0.3);
\draw[-latex] (2.8,-0.3) -- node[below,pos=0.5]{$h$} (4.2,-0.3);

\end{tikzpicture}
\end{center}

\end{Ex}

Let us make some observations before proving our main theorem.

\begin{Lem}\label{lemtempnec2}
Let $\bar{X}$ be an $R$-deformation of a templicial $k$-module $X$. Then we have $k \otimes_R \bar{X}_{\bullet}(a,b)\cong X_{\bullet}(a,b)$ for all $a,b \in S$.
\end{Lem}

\begin{Lem}\label{lemforgetwK}
Let $Y$ be a necklicial $k$-module. If $Y$ is weak Kan as a necklicial $k$-module, it is also weak Kan as a necklicial $R$-module.
\end{Lem}

\begin{proof}
This directly follows from Proposition \ref{propFpres} applied to the forgetful functor $\Mod(k) \lra \Mod(R)$.
\end{proof}

\begin{Thm}\label{thmmainq}
Let $\bar{X}$ be a levelwise flat templicial $R$-module and suppose $X = k \otimes_R \bar{X}$ is a quasi-category in $k$-modules. Then $\bar{X}$ is a quasi-category in $R$-modules.
\end{Thm}

\begin{proof}
	We may assume $I^2 = 0$. By definition we are to show, for $a,b \in S$, that the necklicial $R$-module $\bar{X}_{\bullet}(a,b)$ is weak Kan.
	By the levelwise flatness of $\bar{X}$, we have an induced the short exact sequence 
	$$0 \lra I \otimes_R \bar{X}_{\bullet}(a,b) \lra \bar{X}_{\bullet}(a,b) \lra k \otimes_R \bar{X}_{\bullet}(a,b) = X_{\bullet}(a,b) \lra 0$$
	of necklicial $R$-modules. By Proposition \ref{propextKan}, it suffices to show that both $X_{\bullet}(a,b)$ and $I \otimes_R \bar{X}_{\bullet}(a,b)$ are weak Kan necklicial $R$-modules.
	For this, by Lemma \ref{lemforgetwK}, it suffices that they are both weak Kan as necklicial $k$-modules.
	Now by assumption $X$ is a quasi-category in $k$-modules so $X_{\bullet}(a,b)$ is indeed weak Kan, and by Theorem \ref{thmwingsmain}, so is $I \otimes_k X_{\bullet}(a,b)\cong I \otimes_R \bar{X}_\bullet(a,b)$.
\end{proof}

\subsection{Deformations of deg-projective templicial modules} \label{pardefdegproj}
In this section, we show that the property of being deg-projective lifts under levelwise flat deformation of templicial modules. 

Note that projectivity of modules is preserved by the left adjoint $k \otimes_R -: \Mod(R) \lra \Mod(k)$ of the exact forgetful functor. Conversely, it is well known that projectivity of modules lifts under flat, nilpotent deformation. We will make use of the following slight reinforcement.

\begin{Prop}\label{propproj}
	Let $I^2 = 0$. Let $\bar{P}$ be an $R$-module with $\Tor^{R}_1(k, \bar{P}) = 0$ and such that $P = k \otimes_R \bar{P}$ is projective in $\Mod(k)$. Then $\bar{P}$ is projective in $\Mod(R)$.
\end{Prop}
\begin{proof}
It suffices to show that $\Ext^1_R(\bar{P},M) = 0$ for all $M \in \Mod(R)$. Since
	every such $M$ can be written as an extension of $k$-modules $0 \lra IM \lra M \lra k \otimes_R M \lra 0$, we may assume $M \in \Mod(k)$. 
Take an exact sequence $0 \lra K \lra Q \lra \bar{P} \lra 0$ in which $Q$ is a projective $R$-module. Since $\Tor_1^R(k, \bar{P}) = 0$, there is an induced exact sequence of $k$-modules $0 \lra k \otimes_R K \lra k \otimes_R Q \lra k \otimes_R \bar{P} = P \lra 0$ in which also $k \otimes_R Q$ is projective. The result now follows from the following commutative diagram 
$$
\xymatrix{{\Hom_k(k \otimes_R Q, M)} \ar[d]_{\cong} \ar[r] & {\Hom_k(k \otimes_R K, M)} \ar[r] \ar[d]_{\cong} & {\Ext^1_k(k \otimes_R \bar{P},M)} \ar[r]\ar[d] & 0 \\
{\Hom_R(Q, M)}  \ar[r] & {\Hom_R(K, M)} \ar[r] & {\Ext^1_R(\bar{P},M)} \ar[r] & 0
}
$$
and projectivity of $P = k \otimes_R \bar{P}$ as a $k$-module.
\end{proof}
\begin{Rem}
Alternatively, Proposition \ref{propproj} follows from the Grothendieck spectral sequence associated to $\Hom_R(N,M)\cong \Hom_k(k \otimes_R N, M)$ for $N \in \Mod(R)$ and $M \in \Mod(k)$. See \cite[Prop. 4.7]{lowen2006deformation} for the dual spectral sequences in the context of abelian deformations.	
\end{Rem}

Let $\bar{X}$ be a templicial $R$-module with $X = k \otimes_R \bar{X}$ the induced templicial $k$-module for the functor $k \otimes_R -$ from \eqref{eqtenscolax}.
Consider the associated canonical exact sequences used to express deg-projectivity:

\begin{equation}\label{eqcanbar}
		\xymatrix{{E_{\bar{X}}:} & {\bar{X}^{deg}_n} \ar[r]^{{\can}^{\bar{X}}_n} & {\bar{X}_n} \ar[r] & {\bar{X}^{nd}_n} \ar[r] & 0
	}
\end{equation}
	
\begin{equation}\label{eqcan}
		\xymatrix{{E_X:} & {{X}^{deg}_n} \ar[r]^{{\can}^X_n} & {{X}_n} \ar[r] & {{X}^{nd}_n} \ar[r] & 0 
	}
\end{equation}

\begin{Lem}\label{lemcancan}
	We have $E_X\cong k \otimes_R E_{\bar{X}}$.
\end{Lem}

\begin{proof}
	Immediate from the definition of the canonical sequences, and the fact that $k \otimes_R -$ preserves colimits.
\end{proof}

\begin{Prop}
Let $\bar{X}$ be a deg-projective templicial $R$-module. Then $X = k \otimes_R \bar{X}$ is a deg-projective templicial $k$-module.
\end{Prop}
\begin{proof} 
Consider the canonical exact sequence $E_{\bar{X}}$ from \eqref{eqcanbar}. By assumption, ${\can}^{\bar{X}}_n$ is a monomorphism and $\bar{X}^{nd}_n$ is projective. Consequently, $E_{\bar{X}}$ is a split short exact sequence whence by Lemma \ref{lemcancan} the same holds for $E_X$ from \eqref{eqcan}. A forteriori, $\can^X_n$ is mono. Further, $X^{nd}_n\cong k \otimes_R \bar{X}^{nd}_n$ is projective as desired.
\end{proof}

\begin{Thm}\label{propdegproj}
Let $\bar{X}$ be a levelwise flat templicial $R$-module and suppose $X = k \otimes_R \bar{X}$ is a deg-projective templicial $k$-module. Then $\bar{X}$ is deg-projective.
\end{Thm}

\begin{proof}
We may assume $I^2 = 0$. By the assumption, $E_X$ from \eqref{eqcan} is a split short exact sequence. Consequently, the same holds for $I \otimes_R E_{\bar{X}}\cong I \otimes_k E_X$. For all $n\geq 0$ we obtain the following commutative diagram:
\[
\begin{tikzcd}[]
 & 0 \arrow{d}& & 0 \arrow{d}&\\
  & I\otimes_k X_n^{deg}\arrow[]{r}{}\arrow[hook]{d}{}&\bar{X}_n^{deg} \arrow[two heads]{r}{}\arrow[]{d}& X_n^{deg}\arrow[hook]{d} \arrow{r}& 0\\
   0 \arrow{r}& I \otimes_k X_n \arrow[two heads]{d}{}\arrow[hook]{r}{}& \bar{X}_n \arrow[two heads]{d}{}\arrow[two heads]{r}{}& X_n\arrow{r}\arrow[two heads]{d}{} & 0\\
   & I\otimes_k X_n^{nd}\arrow[]{r}{}\arrow[]{d}{}& \bar{X}_n^{nd} \arrow[two heads]{r}{}\arrow[]{d}& X_n^{nd}\arrow[]{d} \arrow{r}& 0\\
   & 0 & 0 & 0 &\\
\end{tikzcd}
\]
\noindent where the middle row is exact by flatness of $\bar{X}$. We can apply the 4 lemma to the first two rows and conclude that $\can^{\bar{X}}_n$ is a monomorphism and thus the middle column is exact. 
It remains to show that $\bar{X}^{nd}_n$ is projective. By Proposition \ref{propproj}, for this it suffices to show that $\Tor^R_1(k, \bar{X}^{nd}_n) = 0$. To see that this equality holds, note that the long exact $\Tor^R_{\ast}(k,-)$ sequence associated to the middle column features
$$0 \lra \Tor^R_1(k, \bar{X}^{nd}_n) \lra k \otimes_R \bar{X}^{deg}_n \lra k \otimes_R \bar{X} \lra k \otimes_R \bar{X}^{nd}_n \lra 0$$
where we have used levelwise flatness of $\bar{X}$. Comparing with the exact right column (using Lemma \ref{lemcancan}) yields that $\Tor^R_1(k, \bar{X}^{nd}_n) = 0$ as desired.\\

\end{proof}

\begin{Rem}
It is easily seen that in fact all rows and all columns of the above diagram are exact (e.g. by Proposition \ref{propdeglevel}).
\end{Rem}

\printbibliography

\end{document}